\documentclass[10pt,reqno]{amsart}
\usepackage{amsmath}
\usepackage{amssymb}
\usepackage{amsthm}
\newlength{\defbaselineskip}
\newcommand{\setlinespacing}[1]%
           {\setlength{\baselineskip}{#1 \defbaselineskip}}

\numberwithin{equation}{section}

\newtheorem{thm}{Theorem}[section]

\newtheorem{lem}[thm]{Lemma}
\newtheorem{prop}[thm]{Proposition}

\theoremstyle{definition}

\theoremstyle{remark}

\numberwithin{equation}{section}

\begin{document}
\title[The radius of spatial analyticity]
{On the radius of spatial analyticity for defocusing nonlinear Schr\"odinger equations}

\author{Jaeseop Ahn, Jimyeong Kim and Ihyeok Seo}

%\thanks{I. Seo was supported by the NRF grant funded by the Korea government(MSIP) (No. 2017R1C1B5017496).}
\thanks{This research was supported by NRF-2019R1F1A1061316.}


\address{Department of Mathematics, Sungkyunkwan University, Suwon 16419, Republic of Korea}
\email{j.ahn@skku.edu}

%\address{Department of Mathematics, Sungkyunkwan University, Suwon 16419, Republic of Korea}
\email{jimkim@skku.edu}

%\address{Department of Mathematics, Sungkyunkwan University, Suwon 16419, Republic of Korea}
\email{ihseo@skku.edu}

\begin{abstract}
In this paper we study spatial analyticity of solutions to the defocusing nonlinear Schr\"odinger equations
$iu_t + \Delta u = |u|^{p-1}u$, given initial data which is analytic with fixed radius.
It is shown that the uniform radius of spatial analyticity of solutions at later time $t$ cannot decay faster than $1/|t|$ as $|t|\rightarrow\infty$.
This extends the previous work of Tesfahun \cite{Te} for the cubic case $p=3$ to the cases where $p$ is any odd integer greater than $3$.
\end{abstract}

\maketitle

\section{Introduction}\label{sec1}

Consider the Cauchy problem for the defocusing nonlinear Schr\"odinger equations
\begin{equation}\label{NLS}
\begin{cases}
iu_t + \Delta u = |u|^{p-1}u, \\
u(0, x) = u_0(x),
\end{cases}
\end{equation}
where $u:\mathbb{R}^{1+d}\rightarrow\mathbb{C}$ and $p>1$ is an odd integer.
The well-posedness of this Cauchy problem with initial data in Sobolev spaces $H^s(\mathbb{R}^d)$ has been intensively studied.
See for instance \cite{BCF,GV,Ts,CKSTT,FG} and references therein.
In particular, the global well-posedness of \eqref{NLS} is known for
\begin{equation}\label{well}
\begin{cases}
d=1,2\quad\text{with}\quad 1<p<\infty, \\
d=3\quad\text{with}\quad p=3.
\end{cases}
\end{equation}

While the well-posedness in Sobolev spaces is well-understood, much less is known about spatial analyticity of the solutions to the above
Cauchy problem.
Our attention in this paper will be focused on the situation where we consider a real-analytic initial data
with uniform radius of analyticity $\sigma_0>0$,
so there is a holomorphic extension to a complex strip
$$S_{\sigma_0}=\{x+iy:x,y\in\mathbb{R}^d,\,|y_1|,|y_2|,\,\cdots,|y_d|<\sigma_0\}.$$
The question is then whether this property may be continued analytically to a complex strip $S_{\sigma(t)}$ for all later times $t$,
but with a possibly smaller and shrinking radius of analyticity $\sigma(t)>0$.

This type of question was first introduced by Kato and Masuda \cite{KM},
and Bona and Gruji\'{c} \cite{BG} gave an explicit lower bound of the radius $\sigma(t)$ for the Korteweg-de Vries equation.
In fact, it is shown that  $\sigma(t)\geq e^{-ct^2}$ for large $t$ which shows that the radius can decay to zero at most at an exponential rate.
Later, this exponential decay was improved to an algebraic lower bound, $ct^{-12}$, by Bona, Gruji\'{c} and Kalisch \cite{BGK}.
See \cite{SS,Te2,HW} for further refinements.
We also refer the reader to \cite{BGK2,Te,P,ST,S2} for other nonlinear dispersive equations like Schr\"odinger, Klein-Gordon and Dirac-Klein-Gordon equations.

In the present paper we shall work on the Cauchy problem \eqref{NLS}, motivated by an earlier work
on the defocusing cubic nonlinear Schr\"odinger equation, the case $p=3$ in \eqref{NLS}, by Tesfahun \cite{Te}
who gave a lower bound of the radius, $\sigma(t)\geq ct^{-1}$, for large $t$.
It will turn out that it is still possible for the other cases $p>3$ to have the same lower bound.

A nice choice of analytic function space suitable to study spatial analyticity of solution is the Gevrey space
$G^{\sigma,s}(\mathbb{R}^d)$, $\sigma\geq0$, $s\in\mathbb{R}$, with the norm
$$\|f\|_{G^{\sigma,s}}=\big\| e^{\sigma\| D\|} \langle D\rangle^s f\big\|_{L^2},$$
where $D=-i\nabla$ with Fourier symbol $\xi$, $\|\xi\|=\sum_{i=1}^{d} |\xi_i|$ and $\langle\xi\rangle=\sqrt{1+|\xi|^2}$.
In fact, according to the Paley-Wiener theorem\footnote{The proof given for $s=0$ in \cite{K} applies also for $s\in\mathbb{R}$
with some obvious modifications.} (see e.g. \cite{K}, p. 209),
a function $f$ belongs to $G^{\sigma,s}$ with $\sigma>0$
if and only if it is the restriction to the real line of a function $F$ which is holomorphic in the strip
$S_\sigma=\{ x+iy:x,y\in\mathbb{R}^d,\,|y_1|,|y_2|,\,\cdots,|y_d|< \sigma \}$ and satisfies
$\sup_{|y| < \sigma} \| F(x+iy)\|_{H_x^s} <\infty$.
In other words, every function in $G^{\sigma,s}$ with $\sigma>0$ has an analytic extension to the strip $S_\sigma$
(see e.g. \cite{S} for a proof).
This is one of the key properties of the Gevrey space
and shows that the following result gives an algebraic lower bound on the radius of analyticity $\sigma(t)$ of the solution to \eqref{NLS}
as the time $t$ tends to infinity.

\begin{thm}\label{thm1}
Let $d=1,2$.
Let $u$ be the global $C^\infty$ solution of \eqref{NLS} with any odd integer $p>3$
and $u_0 \in G^{\sigma_0,s}(\mathbb{R}^d)$ for some $\sigma_0>0$ and $s\in\mathbb{R}$.
Then, for all $t\in\mathbb{R}$
$$u(t)\in G^{\sigma(t),s}(\mathbb{R}^d)$$
with $\sigma(t)\geq c|t|^{-1}$ as $|t| \rightarrow \infty$.
Here, $c>0$ is a constant depending on $\|u_0\|_{G^{\sigma_0,s}(\mathbb{R}^d)}$, $\sigma_0$, $s$ and $p$.
\end{thm}

Only when $d=1,2$ does the existing well-posedness theory in $H^{s}$ (see \eqref{well}) guarantee
the existence of the global $C^\infty$ solution in the theorem, given initial data $u_0\in G^{\sigma_0,s}$
for all $\sigma_0>0$ and $s\in\mathbb{R}$.
Indeed, observe first that ${G^{0,s}}$ coincides with the Sobolev space $H^s$ and the embeddings
\begin{equation}\label{emb}
G^{\sigma,s}\subset G^{\sigma^\prime,s^\prime}
\end{equation}
for all $0\leq\sigma'<\sigma$ and $s,s'\in\mathbb{R}$.
As a consequence of this embedding with $\sigma'=0$ and the existing well-posedness theory in $H^{s'}$,
the Cauchy problem \eqref{NLS} has a unique smooth solution for all time, given initial data $u_0\in G^{\sigma_0,s}$
for all $\sigma_0>0$ and $s\in\mathbb{R}$.

The outline of this paper is as follows:
In Section \ref{sec2} we present some preliminaries which will be used for the proof of Theorem \ref{thm1}.
In Section \ref{sec3} we obtain some multilinear estimates in Gevrey-Bourgain spaces.
By making use of a contraction argument involving these estimates,
we prove that in a short time interval $0\leq t \leq\delta$ with $\delta>0$ depending on the norm of the initial data,
the radius of analyticity remains strictly positive.
Next, we prove an approximate conservation law, although the conservation of $G^{\sigma_0,1}$-norm of the solution does not hold exactly,
in order to control the growth of the solution in the time interval $[0, \delta]$, measured in the data norm $G^{\sigma_0,1}$.
Section \ref{sec4} is devoted to the proofs of such a local result and an approximate conservation law.
In the final section, Section \ref{sec5}, we finish the proof of Theorem \ref{thm1} by iterating the local result based on the conservation law.

Throughout this paper, the letter $C$ stands for a positive constant which may be different
at each occurrence.

\section{Preliminaries}\label{sec2}

In this section we introduce some function spaces and linear estimates which will be used for the proof of Theorem \ref{thm1}
in later sections.

For $s,b\in\mathbb{R}$, we use  $X^{s,b}=X^{s,b}(\mathbb{R}^{1+d})$ to denote the Bourgain space defined by the norm
$$
\|f\|_{X^{s,b}} =\| \langle\xi\rangle^s\langle\tau+|\xi|^2\rangle^b\widehat{f}(\tau,\xi)\|_{L^2_{\tau,\xi}},
$$
where $\widehat{f}$ denotes the space-time Fourier transform given by
$$
\widehat{f}(\tau,\xi)=\int_{\mathbb{R}^{1+d}} e^{-i(t\tau+x\cdot\xi)}f(t,x) \ dtdx.
$$
The restriction of the Bourgain space, denoted $X^{s,b}_\delta$, to a time slab $(0,\delta )\times\mathbb{R}^d$
is a Banach space when equipped with the norm
$$
\|f\|_{X^{s,b}_\delta}=\inf\big\{\|g\|_{X^{s,b}} : g=f\,\, \text{on}\,\, (0,\delta)\times\mathbb{R}^d\big\}.
$$
We also need to introduce the Gevrey-Bourgain space $X^{\sigma,s,b}=X^{\sigma,s,b}(\mathbb{R}^{1+d})$ defined by the norm
$$
\|f\|_{X^{\sigma,s,b}}=\big\| e^{\sigma\| D\|}f\big\|_{X^{s,b}}.
$$
Its restriction $X^{\sigma,s,b}_\delta$ to a time slab $(0,\delta)\times\mathbb{R}^d$ is defined in a similar way as above,
and when $\sigma=0$ it coincides with the Bourgain space $X^{s,b}$.

The $X^{\sigma,s,b}$-estimates in Lemmas \ref{lem00}, \ref{lem2} and \ref{lem3} follow easily by substitution  $f\rightarrow e^{\sigma\|D\|}f$
using the properties of $X^{s,b}$-spaces and the restrictions thereof.
When $\sigma=0$, the proofs of Lemmas \ref{lem00} and \ref{lem2} can be found in Section 2.6 of \cite{T},
and Lemma \ref{lem3} follows by the argument used for Lemma 3.1 of \cite{CKSTT2}.

\begin{lem}\label{lem00}
Let $\sigma \geq 0$, $s\in\mathbb{R}$ and $b>1/2$.
Then, $X^{\sigma,s,b}\subset C(\mathbb{R},G^{\sigma,s})$ and
$$\sup_{t\in\mathbb{R}}\|f(t)\|_{G^{\sigma,s}}\leq C\|f\|_{X^{\sigma,s,b}},$$
where $C>0$ is a constant depending only on $b$.
\end{lem}

\begin{lem}\label{lem2}
Let $\sigma \geq 0$, $s\in\mathbb{R}$, $-1/2<b<b^\prime<1/2$ and $\delta>0$. Then
$$
\|f\|_{X^{\sigma,s,b}_\delta} \leq C\delta^{b^\prime-b}\|f\|_{X^{\sigma,s,b^\prime}_\delta},
$$
where the constant $C>0$ depends only on $b$ and $b'$.
\end{lem}

\begin{lem}\label{lem3}
Let $\sigma \geq 0$, $s\in\mathbb{R}$, $-1/2<b<1/2$ and $\delta>0$.
Then, for any time interval $I\subset[0,\delta]$,
$$
\|\chi_I f\|_{X^{\sigma,s,b}} \leq C\|f\|_{X^{\sigma,s,b}_\delta},
$$
where $\chi_I(t)$ is the characteristic function of $I$, and the constant $C>0$ depends only on $b$.
\end{lem}

\begin{lem}(See \cite{T})
Let $b>1/2$ and $f\in X^{0,b}$.
Then,
\begin{equation}\label{stricartz_2}
\|f\|_{L^q_tL^r_x} \leq C\|f\|_{X^{0,b}}
\end{equation}
for Schr\"odinger-admissible pair $(q,r)$, i.e.,
$$2\leq q,r\leq \infty,\quad
\frac{2}{q}+\frac{d}{r}=\frac{d}{2}\quad\text{and}\quad (q,r,d)\neq(2,\infty,2).
$$
\end{lem}

Next, consider the linear Cauchy problem for the Schr\"odinger equation
$$
\begin{cases}iu_t+\Delta u = F(t,x), \\
u(0,x)=f(x).
\end{cases}
$$
By Duhamel's principle the solution can be then written as
\begin{equation}\label{DF}
u(t,x)=e^{it\Delta}f(x)-i\int_0^t e^{i(t-s)\Delta}F(s,\cdot)ds,
\end{equation}
where the Fourier multiplier $e^{it\Delta}$ with symbol $e^{-it|\xi|^2}$ is given by
$$e^{it\Delta}f(x)=\frac1{(2\pi)^d}\int_{\mathbb{R}^d}e^{ix\cdot\xi}e^{-it|\xi|^2}\widehat{f}(\xi)d\xi.$$
Then the following is the standard energy estimate in $X_\delta^{s,b}$-spaces (see e.g. \cite{KPV,Te}).
\begin{lem}\label{lem1}
Let $\sigma \geq 0$, $s\in\mathbb{R}$, $1/2<b\leq 1$ and $0<\delta\leq 1$.
Then we have
$$\| e^{it\Delta}f\|_{X^{\sigma,s,b}_\delta} \leq C\|f\|_{G^{\sigma,s}}$$
and
$$\bigg\|\int^t_0 e^{i(t-s)\Delta}F(s,\cdot)ds \bigg\|_{X^{\sigma,s,b}_\delta} \leq C\| F\|_{X^{\sigma,s,b-1}_\delta}.$$
Here the constant $C>0$ depends only on $b$.
\end{lem}

\section{Multilinear estimates in Gevrey-Bourgain spaces}\label{sec3}
In this section we obtain a couple of multilinear estimates in Gevrey-Bourgain spaces, Proposition \ref{lem6},
which will play a key role in obtaining the local well-posedness (Theorem \ref{thm2})
and almost conservation law (Theorem \ref{thm3}) in the next section.
With the aid of Proposition \ref{lem6}, we also deduce two more estimates
which are also important in obtaining the almost conservation law.

From now on, $p$ will always denote an odd integer greater than $3$.
For the sake of brevity, we also let

\begin{equation}\label{s0s1}
s_0 =
\begin{cases}
\frac{p-3}{2(p-1)}&\quad\text{if}\quad d=1,\\
\frac{p-2}{p-1}&\quad\text{if}\quad d=2,
\end{cases}\\
\quad\text{and}\quad
\ s_1=\begin{cases}\frac{p-5}{2(p-2)}&\quad\text{if}\quad d=1,\\
\frac{p-3}{p-2}&\quad\text{if}\quad d=2.
\end{cases}
\end{equation}

\begin{prop}\label{lem6}
Let $d=1,2$, $\sigma \geq 0$ and $b>1/2$.
Then we have
\begin{equation}\label{lemma_estimate_2}
\bigg\|\prod_{j=1}^p U_j \bigg\|_{L^{2}_{t,x}} \leq C_{p,b}\prod_{j=1}^{p-1} \| u_j \|_{X^{s_0,b}}
\| u_p\|_{X^{0,b}},
\end{equation}
\begin{equation}\label{lemma_estimate_1}
\bigg\| \prod_{j=1}^p U_j \bigg\|_{X^{0,-b}} \leq C_{p,b}\prod_{j=1}^{p-2} \| u_j \|_{X^{s_1,b}}
\|u_{p-1}\|_{X^{0,b}} \|u_p\|_{X^{0,b}}
\end{equation}
and
\begin{equation}\label{lemma_estimate_3}
\bigg\|\prod_{j=1}^p U_j \bigg\|_{X^{\sigma,1,0}} \leq C_{p,b}\prod_{j=1}^p \| u_j \|_{X^{\sigma,1,b}},
\end{equation}
where $U_j$ denotes $u_j$ or $\overline{u_j}$.
\end{prop}

\begin{proof}
First we prove \eqref{lemma_estimate_2}.
By duality, it is enough to prove
$$\int(\prod^{p}_{j=1}U_j)\cdot\overline{u_{p+1}}\ dtdx \leq C_{p,b}(\prod^{p-1}_{j=1}\| u_j\|_{X^{s_0,b}})\| u_p\|_{X^{0,b}}\| u_{p+1}\|_{L^2_{t,x}}.$$
By H\"older's inequailty, the Sobolev embedding and then \eqref{stricartz_2}, we get
\begin{align*}
\int (\prod^{p}_{j=1} U_j)\cdot \overline{u_{p+1}} \ dtdx &\leq (\prod^{p-1}_{j=1} \| u_j\|_{L^{\frac{8}{3}(p-1)}_t L^{4(p-1)}_x})\| u_p\|_{{L^8_t}{L^4_x}}\| u_{p+1} \|_{L^2_{t,x}}\\
&\leq  (\prod^{p-1}_{j=1} \| \langle D\rangle^{\frac{p-3}{2(p-1)}}u_j\|_{L^{4(p-1)}_tL^{\frac{4(p-1)}{2p-5}}_x})\| u_p\|_{{L^8_t}{L^4_x}}\| u_{p+1} \|_{L^2_{t,x}}\\
&\leq C_{p,b}(\prod^{p-1}_{j=1} \| u_j \|_{X^{s_0,b}})\| u_p \|_{X^{0,b}}\| u_{p+1} \|_{L^2_{t,x}}
\end{align*}
for $d=1$, and for $d=2$
\begin{align*}
\int (\prod^{p}_{j=1} U_j)\cdot \overline{u_{p+1}} \ dtdx &\leq (\prod^{p-1}_{j=1} \| u_j\|_{L^{4(p-1)}_{t,x}})\| u_p\|_{L^4_{t,x}}\| u_{p+1} \|_{L^2_{t,x}}\\
&\leq  (\prod^{p-1}_{j=1} \| \langle D\rangle^{\frac{p-2}{p-1}}u_j\|_{L^{4(p-1)}_tL^{\frac{4(p-1)}{2p-3}}_x})\| u_p\|_{L^4_{t,x}}\| u_{p+1} \|_{L^2_{t,x}}\\
&\leq C_{p,b}(\prod^{p-1}_{j=1} \| u_j \|_{X^{s_0,b}})\| u_p \|_{X^{0,b}}\| u_{p+1} \|_{L^2_{t,x}},
\end{align*}
as desired.

Next, we prove the second estimate \eqref{lemma_estimate_1}.
By duality we first observe that
\begin{align*}
\quad \| \prod^{p}_{j=1} U_j\|_{X^{0,-b}} &=\| \langle \tau + | \xi |^2 \rangle^{-b}\widehat{\prod^{p}_{j=1}U_j}\|_{L^2_{\tau,\xi}}\\
&=\sup_{\| u^\prime\|_{L^2_{\tau,\xi}}=1}\int \widehat{\prod^{p}_{j=1}U_j\langle} \tau + | \xi |^2 \rangle^{-b}u^\prime \ d\tau d\xi \\
&=\sup_{\| u_{p+1}\|_{X^{0,b}}=1} \int \widehat{\prod^p_{j=1}U_j}\cdot\overline{\widehat{u_{p+1}}} \ d\tau d\xi\ \\
&= \sup_{\| u_{p+1}\|_{X^{0,b}}=1} \int \prod^p_{j=1}U_j\cdot\overline{u_{p+1}} \ dt dx
\end{align*}
where $\overline{\widehat{u_{p+1}}}=\langle \tau + | \xi |^2\rangle^{-b} u^\prime$.
Hence it suffices to prove
$$\int\prod^{p}_{j=1}U_j\cdot\overline{u_{p+1}}\ dtdx \leq C_{p,b}(\prod^{p-2}_{j=1} \| u_j \|_{X^{s_1,b}})(\prod^{p+1}_{j=p-1}\| u_j \|_{X^{0,b}}).$$
Again by H\"older's inequailty, the Sobolev embedding and \eqref{stricartz_2}, we get
\begin{align*}
\int \prod^{p}_{j=1} U_j \cdot \overline{u_{p+1}} \ dt dx &\leq (\prod^{p-2}_{j=1} \| u_j \|_{L^{2(p-2)}_{t,x}})(\prod^{p+1}_{j=p-1} \| u_j \|_{L^{6}_{t,x}}) \\
&\leq (\prod^{p-2}_{j=1} \| \langle D \rangle^{\frac{p-5}{2(p-2)}}u_j \|_{L^{2(p-2)}_t L^{\frac{2(p-2)}{p-4}}_x})(\prod^{p+1}_{j=p-1}\| u_j \|_{L^6_{t,x}})\\
&\leq C_{p,b}(\prod^{p-2}_{j=1} \| u_j \|_{X^{s_1,b}})(\prod^{p+1}_{j=p-1}\| u_j \|_{X^{0,b}})
\end{align*}
for $d=1$, and for $d=2$
\begin{align*}
\int \prod^{p}_{j=1} U_j \cdot \overline{u_{p+1}} \ dt dx &\leq (\prod^{p-2}_{j=1} \| u_j \|_{L^{4(p-2)}_{t,x}})(\prod^{p+1}_{j=p-1} \| u_j \|_{L^{4}_{t,x}}) \\
&\leq (\prod^{p-2}_{j=1} \| \langle D \rangle^{\frac{p-3}{p-2}}u_j \|_{L^{4(p-2)}_t L^{\frac{4(p-2)}{2p-5}}_x})(\prod^{p+1}_{j=p-1}\| u_j \|_{L^4_{t,x}})\\
&\leq C_{p,b}(\prod^{p-2}_{j=1} \| u_j \|_{X^{s_1,b}})(\prod^{p+1}_{j=p-1}\| u_j \|_{X^{0,b}}),
\end{align*}
as desired.

Finally, we prove \eqref{lemma_estimate_3}.
First we consider the case $\sigma=0$. We note that
\begin{align*}
\| f \|_{X^{0,1,0}}^2&=\| \langle\xi\rangle\widehat{f}(\tau,\xi)\|_{L^2_{\tau,\xi}}^2\\
&=\int |\sqrt{1+|\xi |^2}\widehat{f}(\tau,\xi)|^2\ d\tau d\xi\\
&=\int |\widehat{f}(\tau,\xi)|^2\ d\tau d\xi+\int |\xi|^2|\widehat{f}(\tau,\xi)|^2\ d\tau d\xi\\
&=\big\|\widehat{f}(\tau,\xi)\big\|_{L^2_{\tau,\xi}}^2+\big\| |\xi|\widehat{f}(\tau,\xi)\big\|_{L^2_{\tau,\xi}}^2.\\
\end{align*}
Hence it is enough to show that
\begin{equation*}
\bigg\|\prod_{j=1}^p U_j \bigg\|_{L_{t,x}^2} \leq C_{p,b}\prod_{j=1}^p \| u_j \|_{X^{1,b}}
\end{equation*}
and by symmetry
$$\bigg\| \prod^{p-1}_{j=1} U_j\cdot(\partial_{x_1}U_p)\bigg\|_{L^2_{t,x}}\leq C_{p,b}\prod^{p}_{j=1} \| u_j\|_{X^{1,b}},$$
which are direct consequences of \eqref{lemma_estimate_2}.
For instance,
\begin{align*}
\bigg\| \prod^{p-1}_{j=1} U_j\cdot(\partial_{x_1}U_p)\bigg\|_{L^2_{t,x}} &\leq \bigg(\prod^{p-1}_{j=1} \| u_j \|_{X^{s_0,b}}\bigg)
\| \partial_{x_1}u_p\|_{X^{0,b}}\\
&\leq \bigg(\prod^{p-1}_{j=1} \| u_j \|_{X^{1,b}}\bigg)\| u_p\|_{X^{1,b}}.
\end{align*}
Note here that $s_0<1$.
Now we consider the case $\sigma >0$.
Without loss of generality, we may assume $U_k =u_k$ for each $k\in\left\{1,2,\cdots,p\right\}$.
Let $v_j = e^{\sigma\|D\|}u_j$. Then \eqref{lemma_estimate_3} reduces to
\begin{equation*}
\bigg\| e^{\sigma\|D\|}\bigg(\prod^{p}_{j=1} e^{-\sigma\|D\|}v_j\bigg)\bigg\|_{X^{1,0}}\leq C_{p,b}\prod^{p}_{j=1}\| v_j\|_{X^{1,b}}.
\end{equation*}
We first write the space-time Fourier transform of
$$e^{\sigma\|D\|}\bigg(\prod^{p}_{j=1} e^{-\sigma\|D\|}v_j\bigg)$$
as
$$\int_\ast \left\{e^{-\sigma(\sum_{j=1}^p \| \xi_j \| - \| \xi \|)} \right\}\prod_{j=1}^{p}\widehat{v}_j(\tau_j,\xi_j)
\ d\tau_1d\xi_1\cdots d\tau_{p-1}d\xi_{p-1}$$
where we used $\ast$ to denote the conditions $\tau=\sum_{j=1}^p\tau_j$ and $\xi=\sum_{j=1}^p\xi_j$.
In fact, observe that
\begin{align*}
\widehat{v_1v_2v_3}(\tau,\xi) &= \widehat{v_1} * \widehat{v_2v_3}(\tau,\xi)\\
&= \int \widehat{v_1}(\tau_1,\xi_1)\widehat{v_2v_3}(\tau-\tau_1,\xi-\xi_1)\ d\tau_1d\xi_1\\
&= \int \widehat{v_1}(\tau_1,\xi_1)\widehat{v_2}(\tau_2,\xi_2)\widehat{v_3}(\tau-\tau_1-\tau_2,\xi-\xi_1-\xi_2) \ d\tau_2d\xi_2d\tau_1d\xi_1.
\end{align*}
Now letting $\tau_3=\tau-\tau_1-\tau_2$ and $\xi_3=\xi-\xi_1-\xi_2$, the last expression equals to
$$
\int \widehat{v_1}(\tau_1,\xi_1)\widehat{v_2}(\tau_2,\xi_2)\widehat{v_3}(\tau_3,\xi_3) \ d\tau_1d\xi_1d\tau_2d\xi_2.
$$
A similar argument can be made for multiplications of four or more terms.

Let $\widehat{w_j}(\tau,\xi)=| \widehat{v_j}(\tau,\xi)|$. Then using the fact that $e^{\sigma(\| \xi\|-\sum^{p}_{j=1}\| \xi_j\|)} \leq1$,
which follows from the triangle inequality, we obtain
\begin{align*}
\bigg\| e^{\sigma\|D\|}\bigg(\prod^{p}_{j=1} e^{-\sigma\|D\|}v_j\bigg)\bigg\|_{X^{1,0}}
 &\leq \bigg\| \int \langle \xi \rangle \prod^{p}_{j=1} |\widehat{v_j}(\tau_j,\xi_j) | \ d\tau_1d\xi_1\cdots d\tau_{p-1}d\xi_{p-1}\bigg\|_{L^2_{\tau,\xi}}\\
&\leq \bigg\| \int \langle \xi \rangle \prod^{p}_{j=1} \widehat{w_j}(\tau_j,\xi_j) \ d\tau_1d\xi_1\cdots d\tau_{p-1}d\xi_{p-1}\bigg\|_{L^2_{\tau,\xi}}\\
&=\bigg\| \prod^p_{j=1} w_j\bigg\|_{X^{1,0}}.
\end{align*}
Applying the above case $\sigma=0$ to the last expression, we get
\begin{align*}
\bigg\| \prod^p_{j=1} w_j\bigg\|_{X^{1,0}}&\leq C_{p,b}\prod^p_{j=1} \| w_j \|_{X^{1,b}}\\
&=C_{p,b}\prod^p_{j=1} \| v_j \|_{X^{1,b}}.
\end{align*}
Now we have the desired result.
\end{proof}

With the aid of Proposition \ref{lem6}, we deduce two more estimates,
which, along with the function $f$ defined here, will play a crucial role in obtaining the almost conservation law, Theorem \ref{thm3},
in the next section.
\begin{lem}
 Let $d=1,2$ and
\begin{equation}\label{f(v)}
f(v)=-\left\{ | v|^{p-1}v - e^{\sigma \|D\|}\Big(\big| e^{-\sigma \|D\|}v\big|^{p-1}e^{-\sigma \|D\|}v\Big)\right\}.
\end{equation}
For all $b> 1/2$ we then have
\begin{equation} \label{lemma_estimate_4}
\big\| \overline{f(v)}\big\|_{L^{2}_{t,x}}\leq C_{p,b}\sigma \| v \|^{p-1}_{X^{s_0,b}}\| v \|_{X^{1,b}}
\end{equation}
and
\begin{equation}\label{lemma_estimate_5}
\big\| \overline{\triangledown f(v)} \big\|_{X^{0,-b}}\leq C_{p,b}\sigma \| v \|^{p-2}_{X^{s_1,b}}
\| v \|^{2}_{X^{1,b}}.
\end{equation}
\end{lem}

\begin{proof}
We first take the space-time Fourier Transform of $\bar{f}$ to see
\begin{align*}
\Big| \widehat{\overline{f(v)}}(&\tau,\xi)\Big|\\
=&\bigg| \int_\ast \left\{1-e^{-\sigma(\sum_{j=1}^p \|\xi_j \| - \| \xi\|)} \right\}\prod_{j=1}^{\frac{p-1}{2}}\widehat{v}(\tau_j,\xi_j)
\overline{\prod_{j={\frac{p+1}{2}}}^p\widehat{v}(\tau_j,\xi_j) } \ d\tau_1d\xi_1\cdots d\tau_{p-1}d\xi_{p-1}\bigg|
\end{align*}
where we used $\ast$ to denote the conditions $\tau=\sum_{j=1}^p\tau_j$ and $\xi=\sum_{j=1}^p\xi_j$.
By symmetry, we may assume $ \| \xi_1 \| \leq \| \xi_2 \|  \leq \cdots \leq \| \xi_p \| $. Then we have
\begin{align*}
1-e^{-\sigma(\sum_{i=1}^p \|\xi_i \| - \| \xi \|)} &\leq \sigma\bigg(\sum_{i=1}^p \| \xi_i \| - \| \xi \|\bigg)\\
&= \sigma\frac{\big(\sum_{i=1}^p \| \xi_i \|\big)^2 - \| \xi \|^2}{\sum_{i=1}^p \| \xi_i \| + \| \xi \|}\\
&\leq C_p\sigma \frac{\| \xi_{p-1} \| \cdot \| \xi_p \|}{\| \xi_p \|} =C_p\sigma\|\xi_{p-1} \|.
\end{align*}
Consequently,
$$\Big| \widehat{\overline{f(v)}}(\tau,\xi)\Big| \leq\int C_p\sigma\|\xi_{p-1}\| \prod_{j=1}^p|\widehat{v}(\tau_j,\xi_j)| \ d\tau_1d\xi_1\cdots d\tau_{p-1}d\xi_{p-1}.$$
Let $ \widehat{w}_i(\tau_i,\xi_i) = | \widehat{v}(\tau_i,\xi_i) |$.
We then use \eqref{lemma_estimate_2} to obtain
\begin{align*}
\big\| \overline{f(v)} \big\|_{L^2_{t,x}}&= \Big\| \widehat{\overline{f(v)}}(\tau,\xi)\Big\|_{L^2_{\tau,\xi}} \\
&\leq C_p\sigma \bigg\lVert \int \langle \xi_p \rangle\prod_{j=1}^p\widehat{w_j}(\tau_j,\xi_j)\ d\tau_1d\xi_1 \cdots d\tau_{p-1}d\xi_{p-1}\bigg\rVert_{L^2_{\tau,\xi}}\\
&=C_p\sigma \bigg\lVert\prod_{j=1}^{p-1}w_j\cdot \langle D \rangle w_p \bigg\rVert_{L^2_{t,x}}\\
&\leq C_{p,b}\sigma\prod_{j=1}^{p-1}\lVert w_j \rVert_{X^{s_0,b}}\big\lVert \langle D \rangle w_p \big\rVert_{X^{0,b}}\\
&=C_{p,b}\sigma \lVert v \rVert^{p-1}_{X^{s_0,b}}\lVert v \rVert_{X^{1,b}}.
\end{align*}
Here we used the fact that $\|\xi_p\|\leq2|\xi_p|\leq2\langle\xi_p\rangle$ for the first inequality.
Similarly, we use \eqref{lemma_estimate_1} to obtain
\begin{align*}
\big\|\overline{\triangledown f(v)} \big\|_{X^{0,-b}} &= \big\| |\xi| \langle \tau + |\xi|^2 \rangle^{-b}\widehat{\overline{f(v)}}(\tau,\xi)\big\|_{L^2_{\tau,\xi}}\\
&\leq C_p\sigma\bigg\| |\xi| \langle \tau + |\xi|^2 \rangle^{-b}
\int \langle \xi_{p-1} \rangle\prod_{j=1}^p\widehat{w}(\tau_j,\xi_j)d\tau_1d\xi_1 \cdots d\tau_{p-1}d\xi_{p-1}\bigg\|_{L^2_{\tau,\xi}}\\
&\leq C_p\sigma\bigg\| \langle \tau + |\xi|^2 \rangle^{-b} \int \langle \xi_{p-1} \rangle \langle \xi_p \rangle\prod_{j=1}^p\widehat{w}(\tau_j,\xi_j)\ d\tau_1d\xi_1 \cdots d\tau_{p-1}d\xi_{p-1}\bigg\|_{L^2_{\tau,\xi}}\\
&=C_p\sigma\bigg\lVert\prod_{j=1}^{p-2}w_j\cdot \langle D \rangle w_{p-1} \cdot \langle D \rangle w_p \bigg\rVert_{X^{0,-b}}\\
&\leq C_{p,b}\sigma\prod_{j=1}^{p-2}\lVert w_j \rVert_{X^{s_1,b}}\big\lVert \langle D \rangle w_{p-1} \big\rVert_{X^{0,b}} \big\lVert \langle D \rangle w_p \big\rVert_{X^{0,b}}\\
&=C_{p,b}\sigma \lVert v \rVert^{p-2}_{X^{s_1,b}}\lVert v \rVert^{2}_{X^{1,b}}\textrm{.}
\end{align*}
Here we used, for the second inequality, that
$$|\xi|\leq\|\xi\|\leq\sum_{j=1}^p\|\xi_j\|\leq p\|\xi_p\|\leq2p|\xi_p|\leq2p\langle\xi_p\rangle.$$
\end{proof}

\section{Local well-posedness and almost conservation law}\label{sec4}

In this section we shall establish the local well-posedness in Subsection \ref{subsec4.1} and the almost conservation law
in Subsection \ref{subsec4.2} by making use of the multilinear estimates obtained in the previous section.

\subsection{Local well-posedness}\label{subsec4.1}

Based on Picard's iteration in the $X_\delta^{\sigma,1,b}$-space and Lemma \ref{lem00},
we establish the following local well-posedness in $G^{\sigma, 1}$, with a lifespan $\delta > 0$.
Equally the radius of analyticity remains strictly positive in a short time interval $0 \leq t \leq\delta$,
where $\delta > 0$ depends on the norm of the initial data.
We use the notation $A\pm=A\pm\epsilon$ for sufficiently small $\epsilon>0$.

\begin{thm}\label{thm2}
Let $d=1, 2$, $\sigma>0$ and $b=\frac{1}{2}+$.
Then, for any $u_0 \in G^{\sigma,1}$, there exist $\delta > 0$ and a unique solution u of the Cauchy problem \eqref{NLS} on the time interval $[0,\delta]$ such that $u \in C([0,\delta],G^{\sigma,1})$ and the solution depends continuously on the data $u_0$.
Furthermore, we have
$$ \delta = c_0(1+\| u_0 \|_{G^{\sigma,1}})^{-(2(p-1)-)}$$
for some constant $c_0>0$ depending only on $p$,
and the solution u satisfies
\begin{equation}\label{lowe}
\| u \|_{X^{\sigma,1,b}_\delta} \leq C_b\| u_0 \|_{G^{\sigma,1}}.
\end{equation}
\end{thm}

\begin{proof}

Fix $\sigma >0$ and $u_0\in G^{\sigma,1}$.
By Lemma \ref{lem00} we shall employ an iteration argument in the space $X^{\sigma,1,b}_\delta$ instead of $G^{\sigma,1}$.
Let $\left\{ u^{(n)} \right\}^{\infty}_{n=0}$ be the sequence defined by
$$
\begin{cases}
iu^{(0)}_t + \Delta u^{(0)} = 0,\\
u^{(0)}(0) = u_0(x),
\end{cases}
\text{and}\quad
\begin{cases}
iu^{(n)}_t + \Delta u^{(n)} = | u^{(n-1)}|^{p-1}u^{(n-1)},\\
u^{(n)}(0) = u_0(x),
\end{cases}
$$
for $n\in\mathbb{Z}^+$.
Applying \eqref{DF}, we first write
$$u^{(0)}(t,x) = e^{it\Delta}u_0(x)$$
and
$$u^{(n)}(t,x) =e^{it\Delta}u_0(x) - i \int_0^t e^{i(t-s)\Delta}\big(|u^{(n-1)}(s,\cdot)|^{p-1}u^{(n-1)}(s,\cdot)\big) ds.$$
By Lemma \ref{lem1} we have
\begin{equation}\label{0step}
\| u^{(0)}\|_{X^{\sigma,1,b}_\delta}\leq C_b\| u_0 \|_{G^{\sigma,1}},
\end{equation}
and Lemmas \ref{lem1} and \ref{lem2} combined imply
\begin{align}\label{iii}
\nonumber\| u^{(n)} \|_{X^{\sigma,1,b}_\delta} &\leq C_b\| u_0\|_{G^{\sigma,1}} + C_b\big\| | u^{(n-1)}|^{p-1}u^{(n-1)}\big\|_{X^{\sigma,1,b-1}_\delta}\\
\nonumber&\leq C_{b,b'}\| u_0\|_{G^{\sigma,1}} +C_{b,b'}\delta^{b^\prime-b}\big\| | u^{(n-1)}|^{p-1}u^{(n-1)}\big\|_{X^{\sigma,1,b^\prime-1}_\delta}\\
&\leq C_{b,b'}\| u_0\|_{G^{\sigma,1}} +C_{b,b'}\delta^{b^\prime-b}\big\| | u^{(n-1)}|^{p-1}u^{(n-1)}\big\|_{X^{\sigma,1,0}_\delta}
\end{align}
with $1/2<b<b^\prime<1$.
Applying \eqref{lemma_estimate_3} to the second term in the right-hand side of \eqref{iii}, we obtain
\begin{equation}\label{nstep}
\| u^{(n)} \|_{X^{\sigma,1,b}_\delta}\leq C_{p,b,b'}\| u_0\|_{G^{\sigma,1}} +C_{p,b,b'}\delta^{b^\prime-b}\| u^{(n-1)}\|^{p}_{X^{\sigma,1,b}_\delta}.
\end{equation}
By induction together with \eqref{0step} and \eqref{nstep}, it follows that for all $n\geq0$
\begin{equation}\label{proof}
\| u^{(n)}\|_{X^{\sigma,1,b}_\delta} \leq 2C_{p,b,b'}\| u_0\|_{G^{\sigma,1}}
\end{equation}
with a choice of
$$\delta \leq \frac{1}{\big(2\cdot4^{p-1}C_{p,b,b'}^p\| u_0 \|^{p-1}_{G^{\sigma,1}}\big)^{\frac{1}{b^\prime - b}}}.$$
Furthermore, Lemma \ref{lem1} and Lemma \ref{lem2} with the same choice of $\delta$ yield
\begin{align*}
\| u^{(n)}-u^{(n-1)}\|_{X^{\sigma,1,b}_\delta}
&\leq C_b\big\| | u^{(n-1)}|^{p-1}u^{(n-1)} -| u^{(n-2)}|^{p-1}u^{(n-2)}\big\|_{X^{\sigma,1,b-1}_\delta} \\
&\leq C_{b,b'}\delta^{b^\prime-b}\big\| | u^{(n-1)}|^{p-1}u^{(n-1)} -| u^{(n-2)}|^{p-1}u^{(n-2)}\big\|_{X^{\sigma,1,b^\prime-1}_\delta}.
\end{align*}
Since $b'-1<0$ and
\begin{align}\label{imp}
\nonumber| a|^{p-1}a -| b|^{p-1}b
&\leq C_p(| a|^{p-1} + | b|^{p-1})| a-b|\\
&\leq C_p(| a| + | b|)^{p-1}|a-b|,
\end{align}
applying \eqref{lemma_estimate_3}, we now get
\begin{align*}
\| u^{(n)}-&u^{(n-1)}\|_{X^{\sigma,1,b}_\delta}\\
\leq &C_{p,b,b'}\delta^{b^\prime-b}\big(\| u^{(n-1)}\|_{X^{\sigma,1,b}_\delta} + \| u^{(n-2)}\|_{X^{\sigma,1,b}_\delta}\big)^{p-1}
\big(\| u^{(n-1)} - u^{(n-2)}\|_{X^{\sigma,1,b}_\delta}\big).
\end{align*}
By \eqref{proof} and the choice of $\delta$, this implies
$$
\| u^{(n)}-u^{(n-1)}\|_{X^{\sigma,1,b}_\delta}\leq \frac{1}{2}\| u^{(n-1)} - u^{(n-2)}\|_{X^{\sigma,1,b}_\delta}
$$
which guarantees the convergence of the sequence $\left\{ u^{(n)} \right\}^{\infty}_{n=0}$ to a solution $u$
with the bound \eqref{proof}.

Now assume that $u$ and $v$ are solutions to the Cauchy problem $\eqref{NLS}$ for initial data $u_0$ and $v_0$, respectively.
Then similarly as above, again with the same choice of $\delta$ and for any $\delta'$ such that $0<\delta^\prime<\delta$, we have
$$ \| u-v \|_{X^{\sigma,1,b}_{\delta^\prime}} \leq C_b\| u_0 -v_0 \|_{G^{\sigma,1}} + \frac{1}{2}\| u-v \|_{X^{\sigma,1,b}_{\delta^\prime}}$$
provided $\| u_0-v_0\|_{G^{\sigma,1}}$ is sufficiently small, which proves the continuous dependence of the solution on the initial data.

Lastly it remains to show the uniqueness of solutions.
Assume $u, v \in C_tG^{\sigma,1}$ are solutions to $\eqref{NLS}$ for the same initial data $u_0$ and let $w=u-v$. Then $w$ satisfies $iw_t+\Delta w =| u|^{p-1}u -| v|^{p-1}v$. Multiplying both sides by $\bar{w}$ and taking imaginary parts thereon, we have
$$
\textrm{Re}(\bar{w}w_t)+\textrm{Im}(\bar{w}\Delta w )=\textrm{Im}(| u|^{p-1}u\bar{w} -| v|^{p-1}v\bar{w}).
$$
Integrating in $x$ and applying \eqref{imp}, we obtain
\begin{align*}
\frac{d}{dt}\| w(t,x) \|^2_{L^2_x} &\leq \big\|\big(| u|^{p-1}u -| v|^{p-1}v\big)\bar{w}\big\|_{L^1_x}\\
&\leq C\big\|\big(| u|+| v|\big)^{p-1}w^2\big\|_{L^1_x}\\
&\leq C\big\|| u|+| v|\big\|_{L^\infty_x}^{p-1} \| w\|_{L^2_x}^2\\
&\leq C\| w(t,x)\|^2_{L^2_x}.
\end{align*}
Here we used the facts that
$$\frac{d}{dt}| w|^2=2\textrm{Re}(\bar{w}w_t),$$
$$\textrm{Im}\int \bar{w}\Delta w   \ dx=\textrm{Im}\int \nabla w \cdot \overline{\nabla w} \ dx=\textrm{Im}\int | \nabla w |^2 \ dx =0$$
and for all $2\leq q\leq\infty$
$$
G^{\sigma,1} \subseteq H^2 \subseteq L^q.
$$
By Gr\"onwall's inequality, we now conclude that $w=0$.
\end{proof}

\subsection{Almost conservation law}\label{subsec4.2}

We have established the existence of local solutions; we would like to apply the local result repeatedly to cover time intervals of arbitrary length. This, of course, requires some sort of control on the growth of the norm on which the local existence time depends. Observe that a solution to the Cauchy problem \eqref{NLS} satisfies $$M(u(t)):=\| u(t)\|_{L^2_x}^2=M(u(0))$$ and $$E(u(t)):=\|\nabla u(t)\|_{L^2_x}^2+\frac{2}{p+1}\| u(t)\|_{L^{p+1}_x}^{p+1}=E(u(0))$$which are the conservation of mass and energy, respectively. Define a quantity:
$$A_\sigma(t) = \| u(t) \|^{2}_{G_x^{\sigma,1}} + \frac{2}{p+1}\big\| e^{\sigma\| D \|}u(t)\big\|^{p+1}_{L^{p+1}_x}.$$
Then one can easily see $$A_0(t) = A_0(0)\ \textrm{ for all}\ t.$$
This quantity is approximately conservative in the sense that, although it fails to hold for $\sigma>0$, the discrepancy between both sides is bounded as well as the quantity reduces to the conservation of mass and energy in the limit $\sigma\rightarrow 0$. This approximate conservation will allow us (see Section \ref{sec5}) to repeat the local result on successive short-time intervals to reach any target time $T>0$, by adjusting the strip width parameter $\sigma$ according to the size of $T$.
\begin{thm}\label{thm3}
Let $d=1, 2$, and $\delta$ be as in Theorem \ref{thm2}. Then there exists $C_{p,b}>0$ such that for any $\sigma > 0$ and any solution $u \in X^{\sigma,1,b}_{\delta}$ to the Cauchy problem \eqref{NLS} on the time interval $[0,\delta]$, we have the estimate
\begin{equation}\label{acl00}
\sup_{t\in[0,\delta]}A_\sigma(t) \leq A_\sigma (0)+C_p\sigma A^{\frac{p+1}{2}}_{\sigma}(0)(1+A^{\frac{p-1}{2}}_{\sigma}(0)).
\end{equation}
\end{thm}

\begin{proof}
Once we obtain
\begin{equation}\label{acl0}
\sup_{t \in [0,\delta]}A_{\sigma}(t)
\leq  A_\sigma(0) + C_{p,b}\sigma\| u \|^{p-2}_{X^{\sigma,s_0,b}_{\delta}}\| u \|^{3}_{X^{\sigma,1,b}_{\delta}}
\Big(1+\| u\|^{p-1}_{X^{\sigma,1,b}_{\delta}}\Big),
\end{equation}
the desired estimate \eqref{acl00} follows directly from using the bound \eqref{lowe} in Theorem \ref{thm2}.
Note here that $s_0\leq1$ (see \eqref{s0s1}).
We shall prove \eqref{acl0} from now on.

\

\noindent\emph{Step 1.}
Let $0\leq\delta'\leq\delta$.
Setting $v(t,x)=e^{\sigma\|D\|}u(t,x)$ and applying $e^{\sigma\|D\|}$ to \eqref{NLS}, we obtain
\begin{equation}
iv_t + \Delta v = | v |^{p-1}v + f(v)\label{step1}
\end{equation}
where $f(v)$ is as in \eqref{f(v)}.
Multiplying both sides by $\bar{v}$ and taking imaginary parts thereon, we have
\begin{equation*}
\textrm{Re}(\bar{v}v_t)+\textrm{Im}(\bar{v}\Delta v) = \textrm{Im}(\bar{v}f(v)),
\end{equation*}
or equivalently
$$
\partial_t | v |^2 + 2\textrm{Im}(\nabla\cdot(\bar{v}\nabla v))=2\textrm{Im}(\bar{v}f(v))
$$
where we used the facts $\partial_t | v |^2=2\textrm{Re}(\bar{v}v_t)$ and $\textrm{Im}(\bar{v}\Delta v) =\textrm{Im}(\nabla\cdot(\bar{v}\nabla v) -| \nabla v |^2)=\textrm{Im}(\nabla\cdot(\bar{v}\nabla v))$.
We may assume that $v$, $\nabla v$ and $\Delta v$ decay to zero
as $|x|\rightarrow\infty$.\footnote{This property can be shown by approximation using the monotone convergence theorem
and the Riemann-Lebesgue lemma whenever $u\in X_\delta^{\sigma,1,b}$. See the argument in \cite{SS}, p. 1018.}
Using this fact and integrating in space yield
$$
\theoremstyle{definition}\frac{d}{dt}\int_{\mathbb{R}^d} | v |^2 \, dx = 2\textrm{Im}\int_{\mathbb{R}^d} \bar{v}f(v) \, dx,
$$
and subsequently integrating in time over the interval $[0,\delta']$, we have
$$
\int_{\mathbb{R}^d} | v(\delta')|^2 \ dx = \int_{\mathbb{R}^d} | v(0)|^2 \ dx + 2\textrm{Im}\int_{\mathbb{R}^{1+d}}\chi_{[0,\delta']}(t)\bar{v}f(v)\ dtdx.
$$
Now by H\"older's inequality, \eqref{lemma_estimate_4} and Lemma \ref{lem3}, the rightmost integral has the following estimate:
\begin{align*}
\bigg| \int_{\mathbb{R}^{1+d}}\chi_{[0,\delta']}(t)\bar{v}f(v) dtdx\bigg|
&\leq \| v\|_{L^2_t[0,\delta']L^2_x}\big\| \overline{f(v)}\big\|_{L^2_t[0,\delta']L^2_x}\\
&\leq \| v\|_{X^{0,b}_{\delta'}}\cdot C_{p,b}\sigma\|\chi_{[0,\delta']}(t) v\|^{p-1}_{X^{s_0,b}}\|\chi_{[0,\delta']}(t) v\|_{X^{1,b}}\\
&\leq C_{p,b}\sigma\| v\|^p_{X^{s_0,b}_{\delta'}}\| v\|_{X^{1,b}_{\delta'}}.
\end{align*}
Therefore we have
\begin{equation}\label{step1_finish}
\| u(\delta') \|^2_{G^{\sigma,0}} \leq \| u(0) \|^2_{G^{\sigma,0}} +C_{p,b}\sigma\| u\|^p_{X^{\sigma,s_0,b}_{\delta'}}\| u\|_{X^{\sigma,1,b}_{\delta'}}.
\end{equation}

\

\noindent\emph{Step 2.}
Differentiating \eqref{step1} in space we see
$$i\nabla v_t+\Delta\nabla v=\nabla\big(|v|^{p-1}v\big)+\nabla f(v),$$
and multiplying both sides by $\overline{\nabla v}$, we get
$$
\textrm{Re}(\overline{\nabla v} \cdot \nabla v_t) +\textrm{Im}(\overline{\nabla v }\cdot \Delta\nabla v)
=\textrm{Im}\big(\overline{\nabla v} \cdot \nabla (| v|^{p-1}v)+\overline{\nabla v}\cdot\nabla f(v)\big).
$$
Note here that
$\partial_t | \nabla v |^2=2\textrm{Re}(\overline{\nabla v}\cdot\nabla v_t)$,
$\overline{\nabla v }\cdot \Delta\nabla v=\nabla\cdot(\overline{\nabla v}\Delta v)-|\Delta v|^2$ and
\begin{align*}
\overline{\nabla v}\cdot\nabla(| v|^{p-1}v)&=\nabla \cdot(\overline{\nabla v}| v|^{p-1}v)-\overline{\Delta v}| v|^{p-1}v\\
&=\nabla \cdot(\overline{\nabla v}| v|^{p-1}v)-\big(i\bar{v}_t +| v|^{p-1}\bar{v}+\overline{f(v)}\big)| v|^{p-1}v
\end{align*}
using \eqref{step1}.
It then follows that
\begin{align}\label{dfg}
\nonumber\partial_t\big(| \nabla v|^2& + \frac{2}{p+1}| v|^{p+1}\big)+2\textrm{Im}\big(\nabla\cdot(\overline{\nabla v}\Delta v)\big)\\
&=2\textrm{Im}\big(\nabla \cdot(\overline{\nabla v}| v|^{p-1}v)-| v|^{p-1}v\overline{f(v)}+\overline{\nabla v}\cdot\nabla f(v)\big).
\end{align}
Here we also used the fact that
\begin{align*}
\textrm{Im}(i\bar{v}_t| v|^{p-1}v)=\textrm{Re}(\bar{v}_t| v|^{p-1}v)&=\frac{1}{2}(\bar{v}_t| v|^{p-1}v + \overline{\bar{v}_t| v|^{p-1}v})\\
&=\frac1{p+1}\partial_t(\bar{v}^{\frac{p+1}{2}}v^{\frac{p+1}{2}})=\frac1{p+1}\partial_t | v|^{p+1}.
\end{align*}
Similarly as in Step 1, integrating \eqref{dfg} in space yields
$$\frac{d}{dt}\int_{\mathbb{R}^d}|\nabla v|^2+\frac2{p+1}|v|^{p+1}dx
=-2\textrm{Im}\int_{\mathbb{R}^d}|v|^{p-1}v\overline{f(v)}-\overline{\nabla v}\cdot\nabla f(v)\,dx,$$
and then integrating in time over the interval $[0,\delta']$ we have
\begin{align}\label{acl1}
\int_{\mathbb{R}^d} | \nabla v(\delta')|^2+\frac{2}{p+1}| v(\delta')|^{p+1}dx &=\int_{\mathbb{R}^d}| \nabla v(0)|^2+\frac{2}{p+1}| v(0)|^{p+1}dx\\ \nonumber -&2\textrm{Im}\int_{\mathbb{R}^{1+d}} \chi_{[0,\delta']}(t)\big(| v|^{p-1}v\overline{f(v)}-\overline{\nabla v}\cdot\nabla f(v)\big)dxdt.
\end{align}
Now we estimate the second integral on the right-hand side in the above.
By H\"older's inequality, we obtain
$$
\bigg| \int_{\mathbb{R}^{1+d}} \chi_{[0,\delta']}(t)| v|^{p-1}v\overline{f(v)}\, dxdt\bigg|
\leq \big\| |v|^{p-1}v\big\|_{L^2_t[0,\delta']L^2_x}\big\|\overline{f(v)}\big\|_{L^2_t[0,\delta']L^2_x}.
$$
We estimate the above inequality respectively as follows: by \eqref{lemma_estimate_2} and Lemma \ref{lem3} we estimate
\begin{equation}\label{estimate1}
\big\| |v|^{p-1}v\big\|_{L^2_t[0,\delta']L^2_x}\leq C_{p,b}\| \chi[0,\delta'](t)v\|^p_{X^{s_0,b}}\leq C_{p,b}\| v\|^{p}_{X^{s_0,b}_{\delta'}},
\end{equation}
while we easily have
\begin{equation}\label{estimate2}
\| {f(v)}\|_{L^2_t[0,\delta']L^2_x} \leq C_{p,b}\sigma\| v\|^{p-1}_{X^{s_0,b}_{\delta'}}\| v\|_{X^{1,b}_{\delta'}}
\end{equation}
by \eqref{lemma_estimate_4} and Lemma \ref{lem3}.
Combining \eqref{estimate1} with \eqref{estimate2}, we estimate
\begin{equation}\label{acl2}
\bigg| \int_{\mathbb{R}^{1+d}} \chi_{[0,\delta']}(t)| v|^{p-1}v\overline{f(v)} \, dxdt\bigg|
\leq C_{p,b}\sigma \| u\|^{2p-1}_{X^{\sigma,s_0,b}_{\delta'}}\| u\|_{X^{\sigma,1,b}_{\delta'}}.
\end{equation}
By H\"older's inequality, \eqref{lemma_estimate_5} and Lemma \ref{lem3}, similarly as above, we estimate
\begin{align}\label{acl3}
\nonumber\bigg| \int_{\mathbb{R}^{1+d}} \chi_{[0,\delta']}\overline{\nabla v}\cdot\nabla f(v) \,dxdt\bigg|
&\leq \|\nabla v\|_{X^{0,b}_{\delta'}}\| \overline{\nabla f(v)}\|_{X^{0,-b}_{\delta'}}\\
\nonumber&\leq \| v \|_{X^{1,b}_{\delta'}}\cdot C_{p,b}\sigma\| v \|^{p-2}_{X^{s_1,b}_{\delta'}}\| v\|^2_{X^{1,b}_{\delta'}}\\
&\leq C_{p,b}\sigma\| u \|^{p-2}_{X^{\sigma,s_1,b}_{\delta'}}\| u\|^3_{X^{\sigma,1,b}_{\delta'}}.
\end{align}
Therefore, by \eqref{acl1}, \eqref{acl2} and \eqref{acl3}, we get
\begin{align}\label{step2_finish}
\nonumber\| \nabla u(\delta')\|^2_{G^{\sigma,0}}+\frac{2}{p+1}\big\|& e^{\sigma\|D\|}u(\delta')\big\|^{p+1}_{L^{p+1}}\\
&\leq \| \nabla u(0)\|^2_{G^{\sigma,0}}+\frac{2}{p+1}\big\| e^{\sigma\|D\|}u(0)\big\|^{p+1}_{L^{p+1}}\\
\nonumber&+C_{p,b}\sigma\| u\|^{2p-1}_{X^{\sigma,s_0,b}_{\delta'}}\| u\|_{X^{\sigma,1,b}_{\delta'}}+
C_{p,b}\sigma\| u \|^{p-2}_{X^{\sigma,s_1,b}_{\delta'}}\| u\|^3_{X^{\sigma,1,b}_{\delta'}}.
\end{align}

\

\noindent\emph{Step 3.}
Adding \eqref{step1_finish} and \eqref{step2_finish} yields
$$\begin{aligned}
\| u(\delta') &\|^{2}_{G^{\sigma,1}} + \frac{2}{p+1}\big\| e^{\sigma\|D\|}u(\delta')\big\|^{p+1}_{L^{p+1}}\\
\leq &\| u(0) \|^{2}_{G^{\sigma,1}} + \frac{2}{p+1}\big\| e^{\sigma\|D\|}u(0)\big\|^{p+1}_{L^{p+1}} + C_{p,b}\sigma\| u \|^{p-2}_{X^{\sigma,s_0,b}_{\delta'}}\| u \|^{3}_{X^{\sigma,1,b}_{\delta'}}(1+\| u \|^{p-1}_{X^{\sigma,1,b}_{\delta'}})
\end{aligned}$$
for all $0\leq\delta'\leq\delta$.
Here we used the fact that $s_1\leq s_0\leq 1$.
This shows the desired estimate \eqref{acl0}.
\end{proof}

\section{Proof of Theorem \ref{thm1}}\label{sec5}

Following the argument in \cite{Te}, we first consider the case $s=1$.
By the embedding \eqref{emb}, the general case $s\in\mathbb{R}$ will reduce to $s=1$
as shown in the end of this section.

\

\noindent\emph{The case $s=1$.}
Let $u_0=u(0)\in G^{\sigma_0,1}(\mathbb{R}^d)$ for some $\sigma_0>0$ and $\delta$ be as in Theorem \ref{thm2}.
By time reversal symmetry of the equation \eqref{NLS}, we may from now on restrict ourselves to positive times $t\geq0$.
For arbitrarily large $T$, we want to show that the solution $u$ to \eqref{NLS} satisfies
$$u(t)\in G^{\sigma(t),1}\quad\textrm{for all }\, t\in [0,T],$$
where
\begin{equation}\label{label1}
\sigma(t)\geq \frac{c}{T}
\end{equation}
with a constant $c>0$ depending on $\| u_0 \|_{G^{\sigma_0,1}}$, $\sigma_0$, and $p$.

Now fix $T$ arbitrarily large.  It suffices to show
\begin{equation}\label{label2}
\sup_{t\in [0,T]}A_\sigma(t) \leq 2A_{\sigma_0}(0)
\end{equation}
for $\sigma$ satisfying \eqref{label1}; recall that $\| u(t) \|^{2}_{G_x^{\sigma,1}}\leq A_\sigma(t)$,
and the Gagliardo-Nirenberg interpolation inequailty confirms the right-hand side of \eqref{label2} is finite:
\begin{align*}
A_{\sigma_0}(0) &=\| u_0 \|^{2}_{G^{\sigma_0,1}} + \frac{2}{p+1}\big\| e^{\sigma_0\|D\|}u_0\big\|^{p+1}_{L^{p+1}}\\
&\leq \| u_0 \|^{2}_{G^{\sigma_0,1}} + C_{d,p}\big\| \nabla(e^{\sigma_0\|D\|}u_0)\big\|^{\alpha(p+1)}_{L^2}\big\| e^{\sigma_0\|D\|}u_0\big\|^{(1-\alpha)(p+1)}_{L^2}\\
&\leq \| u_0 \|^{2}_{G^{\sigma_0,1}} + C_{d,p}\| \nabla u_0\|^{\alpha(p+1)}_{G^{\sigma_0,0}}\| u_0\|^{(1-\alpha)(p+1)}_{G^{\sigma_0,0}}\\
&<\infty
\end{align*}
where $\alpha = \frac{d(p-1)}{2(p+1)}<1$, which in turn implies $u(t)\in G^{\sigma(t),1}$.

Now it remains to show \eqref{label2}.
Choose $n\in\mathbb{Z}^+$ so that $n\delta\leq T<(n+1)\delta$.
Using induction we shall show for any $k\in\left\{1,2,\cdots,n+1\right\}$ that
\begin{equation}\label{induction1}
\sup_{t\in[0,k\delta]}A_\sigma(t) \leq A_\sigma(0) + 2^pC_p\sigma k A^{\frac{p+1}{2}}_{\sigma_0}(0)\big(1+A^{\frac{p-1}{2}}_{\sigma_0}(0)\big)
\end{equation}
and
\begin{equation}\label{induction2}
\sup_{t\in[0,k\delta]}A_\sigma(t) \leq 2A_{\sigma_0}(0),
\end{equation}
provided $\sigma$ satisfies
\begin{equation}\label{set_sigma}
\frac{2^{p+1}T}{\delta}C_p\sigma A^{\frac{p-1}{2}}_{\sigma_0}(0)\big(1+A^{\frac{p-1}{2}}_{\sigma_0}(0)\big)\leq 1.
\end{equation}

Indeed, for $k=1$, we have from \eqref{acl00} that
\begin{align*}
\sup_{t\in [0,\delta]}A_\sigma(t) &\leq A_\sigma(0) + C_p\sigma A^{\frac{p+1}{2}}_{\sigma}(0)\big(1+A^{\frac{p-1}{2}}_{\sigma}(0)\big)\\
&\leq A_{\sigma}(0) + C_p\sigma A^{\frac{p+1}{2}}_{\sigma_0}(0)\big(1+A^{\frac{p-1}{2}}_{\sigma_0}(0)\big) \\
&\leq 2A_{\sigma_0}(0)
\end{align*}
where we used $A_{\sigma}(0)\leq A_{\sigma_0}(0)$ and $C_p\sigma A^{\frac{p-1}{2}}_{\sigma_0}(0)\big(1+A^{\frac{p-1}{2}}_{\sigma_0}(0)\big)\leq 1$ which is a direct consequence of \eqref{set_sigma}.

Now assume \eqref{induction1} and \eqref{induction2} hold for some $k\in\left\{1,2,\cdots,n\right\}.$
Applying \eqref{acl00}, \eqref{induction2} and \eqref{induction1}, we then have
\begin{align*}
\sup_{t\in[k\delta,(k+1)\delta]}A_\sigma(t)&\leq A_{\sigma}(k\delta) + C_p\sigma A^{\frac{p+1}{2}}_\sigma(k\delta)\big(1+A^{\frac{p-1}{2}}_\sigma(k\delta)\big)\\
&\leq A_{\sigma}(k\delta) + 2^pC_p\sigma A^{\frac{p+1}{2}}_{\sigma_0}(0)\big(1+A^{\frac{p-1}{2}}_{\sigma_0}(0)\big)\\
&\leq A_{\sigma}(0) + 2^pC_p\sigma(k+1) A^{\frac{p+1}{2}}_{\sigma_0}(0)\big(1+A^{\frac{p-1}{2}}_{\sigma_0}(0)\big).
\end{align*}
Combining this with the induction hypothesis \eqref{induction1} for $k$, we get
\begin{equation}\label{label4}
\sup_{t\in[0,(k+1)\delta]}A_\sigma(t) \leq A_\sigma(0) + 2^pC_p\sigma (k+1) A^{\frac{p+1}{2}}_{\sigma_0}(0)\big(1+A^{\frac{p-1}{2}}_{\sigma_0}(0)\big)
\end{equation}
which proves \eqref{induction1} for $k+1$.
Since $k+1\leq n+1\leq T/\delta+1\leq2T/\delta$, from \eqref{set_sigma} we also get
$$
2^pC_p\sigma (k+1) A^{\frac{p-1}{2}}_{\sigma_0}(0)\big(1+A^{\frac{p-1}{2}}_{\sigma_0}(0)\big)
\leq\frac{2^{p+1}T}{\delta}C_p\sigma A^{\frac{p-1}{2}}_{\sigma_0}(0)\big(1+A^{\frac{p-1}{2}}_{\sigma_0}(0)\big)\leq1
$$
which, along with \eqref{label4}, proves \eqref{induction2} for $k+1$.

Finally, the condition \eqref{set_sigma} is satisfied for
\begin{equation}\label{label3}
\sigma=\frac{\delta}{2^{p+1}{C_pA^{\frac{p-1}{2}}_{\sigma_0}(0)(1+A^{\frac{p-1}{2}}_{\sigma_0}(0))}}\cdot\frac{1}{T}.
\end{equation} Here the constant $c$ in \eqref{label1} can be given as $$c=\frac{\delta}{2^{p+1}{C_pA^{\frac{p-1}{2}}_{\sigma_0}(0)(1+A^{\frac{p-1}{2}}_{\sigma_0}(0))}}.$$
Note that this $c$ only depends on $\| u_0 \|_{G^{\sigma_0,1}}$, $\sigma_0$ and $p$.

\

\noindent\emph{The case $s\in\mathbb{R}$.}
Recall that \eqref{emb} states
$$
G^{\sigma,s}\subset G^{\sigma^\prime,s^\prime}\ \textrm{ for all }\ \sigma>\sigma'\geq 0\ \textrm{ and }\ s,s'\in\mathbb{R}\textrm{.}
$$
For any $s\in\mathbb{R}$ we use this embedding to get
$$
u_0\in G^{\sigma_0,s}\subset G^{\sigma_0/2,1}.
$$
From the local theory there is a $\delta=\delta(A_{\sigma_0/2}(0))$ such that
$$
u(t)\in G^{\sigma_0/2,1}\quad\textrm{for }\ 0\leq t\leq\delta\textrm{.}
$$
Similarly as in the case $s=1$, for $T$ fixed greater than $\delta$, we have $u(t)\in G^{\sigma^\prime,1}$ for $t\in[0,T]$
and $\sigma^\prime\geq c/T$ with $c>0$ depending on $\| u_0 \|_{G^{\sigma_0/2,1}}$, $\sigma_0$ and $p$.
Applying the embedding again, we conclude
$$
u(t)\in G^{\sigma,s}\quad\text{for}\,\ t\in [0,T]
$$
where $\sigma =\sigma^\prime/2$.

\

\noindent\textbf{Acknowledgment.} The authors would like to thank the anonymous referee for valuable comments on the Paley-Wiener theorem.

\end{document}